\documentclass{siamart1116}

\usepackage{amsfonts}
\usepackage{graphicx}
\usepackage[]{amsmath}
\usepackage[]{microtype}
\usepackage[sort]{natbib}
\allowdisplaybreaks

\numberwithin{theorem}{section}

\newcommand{\TheTitle}{Stochastic Control and Differential Games with Path-Dependent Influence of
Controls on Dynamics and Running Cost} 
\newcommand{\TheAuthors}{Y. F. Saporito}

\headers{Stoch. Control and Differ. Games with Path-Depend. of Controls}{\TheAuthors}

\title{{\TheTitle}}

\author{
  Yuri F. Saporito\thanks{Escola de Matem\'atica Aplicada (EMAp), Funda\c{c}\~ao Get\'ulio Vargas (FGV), Rio de Janeiro, Brazil
    (\email{yuri.saporito@fgv.br}, \url{http://www.yurisaporito.com}).}
}

\newcommand{\ds}{\displaystyle} 

\newcommand{\eps}{\varepsilon} 
\newcommand{\al}{\alpha}

\newcommand{\bE}{\mathbb{E}}

\newcommand{\cF}{\mathcal{F}}
\newcommand{\cA}{\mathcal{A}}

\newcommand{\bP}{\mathbb{P}}
\newcommand{\bA}{\mathbb{A}}

\newcommand{\bC}{\mathbb{C}}
\newcommand{\bN}{\mathbb{N}}
\newcommand{\bR}{\mathbb{R}}
\newcommand{\bS}{\mathbb{S}}

\def\hyph{-\penalty0\hskip0pt\relax}

\ifpdf
\hypersetup{           
  pdftitle={\TheTitle},          
  pdfauthor={\TheAuthors}
}
\fi

\begin{document}

\maketitle

% REQUIRED
\begin{abstract}
In this paper we consider the functional It\^o calculus framework to find a path\hyph dependent version of the Hamilton\hyph Jacobi\hyph Bellman equation for stochastic control problems that feature dynamics and running cost that \textit{depend on the path of the control}. We also prove a Dynamic Programming Principle for such problems. We apply our results to path\hyph dependence of the delay type. We further study Stochastic Differential Games in this context.
\end{abstract}

% REQUIRED

% REQUIRED

\section{Introduction}
Stochastic control problems and differential games appear naturally in various areas of applications. Portfolio allocation, investment\hyph consumption utility maximization, hedging in incomplete markets and real options are some important examples in Finance and Economics. See, for instance, \cite{pham_book} and \cite{carmona_control}. The standard case deals with a controlled diffusion
\begin{align*}
\left\{\begin{array}{l}
dx_s^{t, y, \al} = b(s, x_s, \alpha_s) ds + \sigma(s, x_s, \alpha_s)dw_s, \mbox{ if } s > t,\\ \\
x_t^{t, y, \al} = y,
\end{array}\right.
\end{align*}
and a cost functional
\begin{align*}
J(t, y, \al) = \bE\left[g(x_T^{t, y, \al}) + \int_t^T f(s, x_s^{t, y, \al}, \alpha_s) ds \right],
\end{align*}
where $\al$ is a admissible control and $g$ and $f$ are suitable functions. The quantity of interest here is the value function:
\begin{align*}
V(t, y) = \inf_{\al \in \bA } J(t, y, \al),
\end{align*}
where $\bA$ is the set of admissible controls. 

Two very important results on Stochastic Control are the Dynamic Programming Principle (DPP) and the Verification Theorem for the related Hamilton\hyph Jacobi\hyph Bellman (HJB) equation. The main contribution of our paper is to extend the DPP and the HJB to stochastic control problems that feature dynamics and running cost that depend on the path of the control $\alpha$. The main example to have in mind is the following delayed-control diffusion
\begin{align}
dx_s^{t, y, \alpha} = (\alpha_s - \alpha_{s - \tau})dt + \sigma dw_s, \quad x_t^{t,y,\alpha} =y, \label{eq:simple_example}
\end{align}
for a non-random, fixed $\tau > 0$. In this case, we say the control problem exhibits \textit{path\hyph dependence in the control}. This is not to be confused with the Closed-Loop Perfect State (CLPS) controls, which means that the control $\alpha$ is progressively\hyph measurable with respect to the filtration generated by the process $x$.

Stochastic control has been already extended to consider path\hyph dependence of dynamics and running cost with respect to the state variable $x$, see, for example, \cite{fournie_cont_thesis}, \cite{fito_hjb} and \cite{fito_hjb_ji}. We say the control problem in this case exihbits \textit{path\hyph depndence in the state variable}. Functional It\^o calculus was also applied to the stochastic control problem of portfolio optimization with bounded memory in \cite{fito_stochastic_portfolio_delay}. Furthermore, the theory was applied to zero\hyph sum stochastic differential games in \cite{fito_zhang_9_game}. However, these references do not deal with path\hyph dependent influence of the control, only path\hyph dependence in the state of the system. This generalization is fundamentally different from the one pursued in our paper, which will become clear in the sections to follow, see Remark \ref{rmk:x_control}.

General path\hyph dependent effect of the control in the dynamics and the running cost are still incipient in theory and applications of stochastic control and differential games. This is very likely related to the lack of theoretical tools to deal with such objects in an appropriate way. We hope this work will provide a useful framework. 

For example, in \cite{gozzi_delay_1} and \cite{ gozzi_delay_2_I, gozzi_delay_2_II}, the authors considered a class of problems that exhibit a particular type of path\hyph dependence in the control, namely delayed controls. The method implemented there is a classical infinite\hyph dimensional analysis and they derived an infinite\hyph dimensional HJB equation. However, their method is strongly related to the delay\hyph type of path\hyph dependence. Additionally, we forward the reader to the following articles \cite{delay_control_huang, delay_control_li, delay_control_alekal}. These results were recently applied to stochastic games in \cite{fouque_systemic_delay}. 

Our approach uses the functional It\^o calculus framework, introduced by Bruno Dupire in the seminal paper \cite{fito_dupire}, which allows us to consider more general path\hyph dependent structures. Although our method could be also seen as an infinite\hyph dimensional analysis, it is rather different than the one applied in \cite{gozzi_delay_1} and \cite{gozzi_delay_2_I, gozzi_delay_2_II}. Our method delivers an HJB equation that can be applied to virtually any path\hyph dependent structure in the control and it could be formulated in the deterministic case as well. Our assumptions are mainly related to the well\hyph posedness of the optimal control problem (smoothness, measurability and integrability). Additionally, our method could be applied to delay of the type of Equation (\ref{eq:simple_example}) with no additional difficulty, which is not the case of the method derived in the aforesaid references. See Section \ref{sec:example_control} for more details.

The structure of the paper is as follows. We finish this introduction with the main definitions and results of functional It\^o calculus. In Section \ref{sec:control}, we introduce the problem we are considering and derive the main results of our work: the DPP in Theorem \ref{thm:dpp} and the Verification Theorem for the path\hyph dependent HJB equation in Theorem \ref{thm:hjb_verification}. An example is analyzed in Section \ref{sec:example_control}. Additionally, in Section \ref{sec:games}, we briefly study stochastic differential games with path\hyph dependence effects of the control in the dynamics and running cost.

\subsection{A Crash Course in Functional It\^o Calculus}\label{sec:fito_intro}

The important notions of the functional It\^o calculus framework will be introduced in this section. For more details and results, we forward the reader to \cite{fito_dupire, rama_cont_fito_change_variable}.

We start by fixing a time horizon $T > 0$. Denote $\Lambda_t^n$ the space of c\`adl\`ag paths in $[0,t]$ taking values in $\bR^n$ and define $\Lambda^n = \bigcup_{t \in [0,T]} \Lambda_t^n$ and $\Lambda^{n \times k} = \bigcup_{t \in [0,T]} \Lambda_t^n \times \Lambda_t^k$. Elements of $\Lambda^{n \times k}$ are two paths taking values in $\bR^n$ and $\bR^k$, respectively, with the same time interval as domain. When it is not necessary to distinguish the dimensions of these spaces, we will use the notation $\Lambda$.

Moreover, when considering examples with delay, one could consider $\bigcup_{t \in [-\tau,T]} \Lambda_t$, where $\tau$ is the largest possible value for the delay. In the examples studied here, we will assume that any path at negative time is zero. This does not increase the difficulty in our calculations and could be easily relaxed.

Capital letters will denote elements of $\Lambda$ (i.e. paths) and lower-case letters will denote spot value of paths. In symbols, $Y_t \in \Lambda$ means $Y_t \in \Lambda_t$ and $y_s = Y_t(s)$, for $s \leq t$.

A functional is any function $f: \Lambda \longrightarrow \bR$. For such objects, we define, when the limits exist, the time and space functional derivatives, respectively, as
\begin{align}
\Delta_t f(Y_t) &= \lim_{\delta t \to 0^+} \frac{f(Y_{t,\delta t}) - f(Y_t)}{\delta t}, \label{eq:time_deriv}\\
\Delta_x f(Y_t) &= \lim_{h \to 0} \frac{f(Y_t^h) - f(Y_t)}{h}, \label{eq:space_deriv}
\end{align}
where
\begin{align*}
Y_{t,\delta t}(u) &= \left\{
\begin{array}{ll}
  y_u,  &\mbox{ if } \quad 0 \leq u \leq t, \\
  y_t,  &\mbox{ if } \quad t \leq u \leq t + \delta t,
\end{array}
\right. \\
Y_t^h(u) &= \left\{
\begin{array}{ll}
  y_u,      &\mbox{ if } \quad 0 \leq u < t, \\
  y_t + h,  &\mbox{ if } \quad u = t,
\end{array}
\right.
\end{align*}
see Figures \ref{fig:flat} and \ref{fig:bump}. In the case when the path $Y_t$ lies in a multidimensional space, the path deformations above are understood as follows: the flat extension is applied to all dimension jointly and equally and the bump is applied to each dimension individually.

\begin{figure}[h!]
\centering
  \begin{minipage}[b]{0.45\linewidth}
    \centering
    \includegraphics[width=0.5\linewidth]{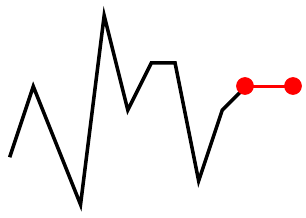}
    \caption{Flat extension of a path.}
    \label{fig:flat}
  \end{minipage}
  \begin{minipage}[b]{0.45\linewidth}
    \centering
    \includegraphics[width=0.5\linewidth]{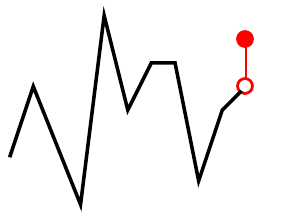}
    \caption{Bumped path.}
    \label{fig:bump}
  \end{minipage}
\end{figure}

We consider here continuity of functionals as the usual continuity in metric spaces with respect to the metric:
\begin{align*}
d_{\Lambda}(Y_t,Z_s) = \| Y_{t,s-t} - Z_s\|_{\infty} + |s -t|,
\end{align*}
where, without loss of generality, we are assuming $s \geq t$, and
\begin{align*}
\|Y_t\|_{\infty} = \sup_{u \in [0,t]} |y_u|.
\end{align*}
The norm $| \cdot |$ is the usual Euclidean norm in the appropriate Euclidian space, depending on the dimension of the path being considered. This continuity notion could be relaxed, see, for instance, \cite{fito_extension_ito_formula}.

Moreover, we say a functional $f$ is \textit{boundedness\hyph preserving} if, for every compact set $K \subset \bR^d$, there exists a constant $C$ such that $|f(Y_t)| \leq C$, for every path $Y_t$ satisfying $Y_t([0,t]) = \{y \in \bR^d \ ; \ Y_t(s) = y \mbox{ for some } s \in [0,t]\} \subset K$.

A functional $f: \Lambda \longrightarrow \bR$ is said to belong to $\bC^{1,2}$ if it is $\Lambda$\hyph continuous, bounded\-ness\hyph preserving and it has $\Lambda$\hyph continuous, boundedness\hyph preserving derivatives $\Delta_t f$, $\Delta_x f$ and $\Delta_{xx} f$. Here, clearly, $\Delta_{xx} = \Delta_x \Delta_x$.

The It\^o formula can be generalized to this framework. The proof can be found in \cite{fito_dupire}. We start by fixing a probability space $(\Omega, \cF, \bP)$.
\begin{theorem}[Functional It\^o Formula; \cite{fito_dupire}]\label{thm:fif}
Let $x$ be a continuous semimartingale and $f \in \bC^{1,2}$. Then, for any $t \in [0,T]$,
\begin{align*}
f(X_t) = f(X_0) + \int_0^t \Delta_t f(X_s) ds + \int_0^t \Delta_x f(X_s) dx_s + \frac{1}{2} \int_0^t \Delta_{xx} f(X_s) d\langle x\rangle_s \quad \bP\mbox{-a.s.}
\end{align*}
\end{theorem}

\section{Main results}\label{sec:main}
\subsection{Stochastic Control with Path\hyph Dependent Controls}\label{sec:control}
We suggest the reader to always keep in mind this example:
\begin{align*}
dx_s^{t, y, \alpha} = (\alpha_s - \alpha_{s - \tau})dt + \sigma dw_s, \quad x_t^{t,y,\alpha} =y.
\end{align*}

Consider a $d$\hyph dimensional Brownian motion $(w_t)_{t \in [0,T]}$ on $(\Omega, \cF, \bP)$ and a filtration $(\cF_t)_{t \in [0,T]}$ in this space, satisfying the usual conditions, to which the Brownian motion $(w_t)_{t \in [0,T]}$ is adapted. One could assume that $(\cF_t)_{t \in [0,T]}$ is the augmented natural filtration of $w$. The set of admissible controls $\bA(\cF_t)$, or just $\bA$, is the space of $\cF_t$\hyph progressively measurable, \textit{c\`adl\`ag} processes in $L^2(\Omega \times [0,T])$ taking value in some subset $\cA \subset \bR^k$. Additional restrictions on $\bA$ will be assumed.

We will consider the following path\hyph dependent controlled diffusion dynamics for $x$:
\begin{align}\label{eq:sde_control}
\left\{\begin{array}{l}
dx_s^{Y_t, \al} = b(X_s^{Y_t, \al}, A_s) ds + \sigma(X_s^{Y_t, \al}, A_s)dw_s, \mbox{ if } s > t,\\ \\
X_t^{Y_t, \al} = Y_t,
\end{array}\right.
\end{align}
where $A_s = (\alpha_t)_{t \in [0,s]}$, i.e. the path of the control $\alpha \in \bA$ up to time $s$, $b : \Lambda^{n \times k} \longrightarrow \bR^n$ and $\sigma : \Lambda^{n \times k}  \longrightarrow \bR^{n \times d}$, with $\bR^{n \times d}$ denoting the space of $n\times d$ matrices. Notice that we are allowing for path\hyph dependence of $b$ and $\sigma$ on the state system, $x$, and on the control, $\alpha$. 

To guarantee existence and uniqueness of strong solutions, we assume there exists a constant $K > 0$ such that
\begin{align*}
\left\{\begin{array}{l}
|b(Y_s, Z_s) - b(Y'_s, Z_s)| + |\sigma(Y_s, Z_s) - \sigma(Y'_s, Z_s)| \leq K \|Y_s - Y'_s\|_{\infty}, \\ \\
|b(Y_s, Z_s)| + |\sigma(Y_s, Z_s)| \leq K\left(1 + |s| + \|Y_s\|_{\infty}\right),
\end{array}\right.
\end{align*}
for all $s \geq t$, $(Y_s, Z_s), (Y_s', Z_s) \in \Lambda^{n \times k}$. These assumptions could be weaken, but it is outside the scope of this work.

Moreover, we consider the following class of cost functionals $J : \Lambda^n \times \bA \longrightarrow \bR$:
\begin{align}\label{eq:cost_functional}
J(Y_t, \al) = \bE\left[g(X_T^{Y_t, \al}) + \int_t^T f(X_s^{Y_t, \al}, A_s) ds \right],
\end{align}
where $g: \Lambda_T^n \longrightarrow \bR$ and $f : \Lambda^{n \times k} \longrightarrow \bR$ satisfy certain measurability and integrability conditions. Notice that $J(Y_T, \al) = g(Y_T)$. We additionally assume that the admissible controls in $\bA$ satisfy straightforward integrability conditions depending on the functionals $b$, $\sigma$ and $f$ so that Equations (\ref{eq:sde_control}) and (\ref{eq:cost_functional}) are well\hyph defined.

We will use the following notation: 
\begin{align}
(Z_t \otimes \al)(s) = \left\{\begin{array}{l}
z_s, \mbox{ if } s < t,\\ \\
\alpha_s, \mbox{ if } s \geq t.
\end{array}\right. \label{eq:otimes}
\end{align}
The path $Z_t \otimes \al$ is equal $Z$ up to time $t$ (excluding it) and then follows the control $\alpha$. Notice that, for any $\alpha \in \bA$ and $Z_t \in \Lambda$, the control $Z_t \otimes \alpha$ is admissible.

We then define the value functional $V : \Lambda^{n \times k} \longrightarrow \bR$:
\begin{align*}
V(Y_t, Z_t) = \inf_{\al \in \bA} J(Y_t, Z_t \otimes \al).
\end{align*}

\begin{remark}[C\`adl\`ag Controls]
The framework we are considering requires the additional assumption that the control is c\`adl\`ag, as it was stated above in the definition of $\bA$. This is inherent of the functional It\^o calculus theory and it allows us to use this technique to analyze these complex stochastic control problems we are considering here. From an application point-of-view, this restriction is not very strong as one would usually restrict even further the space of admissible controls. Although outside the scope of this paper, one could analyze whether the value function considered here is the same as the one for the more general class of progressively-measurable controls.
\end{remark}

We now state and prove the Dynamic Programming Principle for the control problem being considering.

\begin{theorem}[Dynamic Programming Principle (DPP)]\label{thm:dpp}
For any $u \in [t,T]$,
\begin{align*}
V(Y_t, Z_t) = \inf_{\al \in \bA} &\bE\left[V(X_u^{Y_t, Z_t \otimes \al}, (Z_t \otimes \al)_u) + \int_t^u f(X_s^{Y_t, Z_t \otimes \al}, (Z_t \otimes \al)_s) ds \right].
\end{align*}
\end{theorem}

\begin{proof}
The proof follows the same steps as in the path\hyph independent case, since all the coefficients are still adapted. We follow the structure of the proof in \cite{pham_book}.

Firstly, notice that, for any $\al \in \bA$ and $t \leq u \leq s \leq T$, we have the following equivalence of paths
\begin{align*}
X_s^{Y_t, \al} = X_s^{X_u^{Y_t,\al}, \al}.
\end{align*}
Then,
\begin{align*}
J(Y_t, \al) = \bE&\left[g(X_T^{X_u^{Y_t, \al},\al}) + \int_t^u f(X_s^{Y_t, \al},  A_s) ds + \int_u^T f(X_s^{X_u^{Y_t, \al}, \al},  A_s) ds\right],
\end{align*}
and conditioning on the path $X_u^{Y_t, \al}$, we find
\begin{align}
J(Y_t, \al) = \bE\left[J(X_u^{Y_t, \al}, \al) + \int_t^u f(X_s^{Y_t, \al},  A_s) ds\right].\label{eq:dpp_proof}
\end{align}
From this and choosing the control $\al$ to be $Z_t \otimes \al$, it is clear that
\begin{align*}
J(Y_t, Z_t \otimes \al) \geq \bE&\left[V(X_u^{Y_t, Z_t \otimes \al}, (Z_t \otimes \al)_u) + \int_t^u f(X_s^{Y_t, Z_t \otimes \al},  (Z_t \otimes \al)_s) ds\right].
\end{align*}
Taking the infimum with respect to $\al \in \bA$, we find
\begin{align*}
V(Y_t,Z_t) \geq \inf_{\al \in \bA} \bE&\left[V(X_u^{Y_t, Z_t \otimes \al}, (Z_t \otimes \al)_u) + \int_t^u f(X_s^{Y_t, Z_t \otimes \al},  (Z_t \otimes \al)_s) ds\right].
\end{align*}

To prove the opposite inequality, fix $\al \in \bA$ and $u \in [t,T]$. Then, for any $\eps > 0$, there exists $\al^\eps \in \bA$ such that
\begin{align*}
V(X_u^{Y_t, Z_t \otimes \al}, (Z_t \otimes \al)_u) + \eps \geq J(X_u^{Y_t, Z_t \otimes \al}, (Z_t \otimes \al)_u \otimes \al^\eps).
\end{align*}
It can be shown by the Measurable Selection Theorem (see, for example, \cite{soner_touzi_2002}) that $\al^* = A_u \otimes \al^\eps$ belongs to $\bA$ (i.e. it is progressively measurable). Since $Z_t \otimes \al^* = (Z_t \otimes \al)_u \otimes \al^\eps$, by Equation (\ref{eq:dpp_proof}), and $X_u^{Y_t, Z_t \otimes \al^*} = X_u^{Y_t, Z_t \otimes \al}$, we find
\begin{align*}
&V(Y_t, Z_t) \leq J(Y_t, Z_t \otimes \al^*) \\
&=  \bE\left[\int_t^u f(X_s^{Y_t, Z_t \otimes \al}, (Z_t \otimes \al)_s) ds + J(X_u^{Y_t, Z_t \otimes \al}, Z_t \otimes \al^*)\right]\\
&\leq \bE\left[\int_t^u f(X_s^{Y_t, Z_t \otimes \al}, (Z_t \otimes \al)_s) ds + V(X_u^{Y_t, Z_t \otimes \al}, (Z_t \otimes \al)_u)\right] + \eps,
\end{align*}
which implies, by the fact $\al \in \bA$ and $\eps > 0$ are arbitrary, that
\begin{align*}
V(Y_t, Z_t) \leq \inf_{\al \in \bA} \bE&\left[V(X_u^{Y_t, Z_t \otimes\al}, (Z_t \otimes \al)_u) + \int_t^u f(X_s^{Y_t, Z_t \otimes \al}, (Z_t \otimes \al)_s) ds\right],
\end{align*}
from where the final result follows.
\end{proof}

\subsection{The Path\hyph Dependent Hamilton\hyph Jacobi\hyph Bellman Equation}

In this section we will state the HJB equation related to our control problem and also prove a verification theorem for such equation. In the framework of the functional It\^o calculus, this type of equation is called path\hyph dependent Partial Differential Equation, PPDE. See for example, \cite{fito_touzi_ppde1, fito_touzi_ppde2, fito_zhang_1}.

We start by defining the Hamiltonian $H : \Lambda^{n \times k} \times \bR^n \times \bS^n \times \cA \longrightarrow \bR$:
\begin{align*}
H(Y_t, Z_t, p, \gamma, a) = \frac{1}{2} \sigma \sigma^T(Y_t, Z_t^{a-z_t}) : \gamma + b(Y_t, Z_t^{a-z_t}) \cdot p + f(Y_t, Z_t^{a-z_t}),
\end{align*}
and the \textit{modified} Hamiltonian $\widehat{H}: \Lambda^{n \times k} \times \bR^{\Lambda^k} \times\bR^n \times \bS^n \longrightarrow \bR$:
\begin{align}
\widehat{H}(Y_t, Z_t, q, p, \gamma) = \inf_{a \in \cA} \left\{q(Z_t^{a-z_t}) + H(Y_t, Z_t, p, \gamma, a) \right\}\label{eq:modified_hamiltonian}
\end{align}
The symbol $\bR^{\Lambda^k}$ denotes the space of functionals $\Lambda^k \longrightarrow \bR$ and $\bS^n$ is the space of $n \times n$ symmetric matrices. Notice that $Z_t^{a-z_t}$ is changing the last value of the control $Z_t$ to $a$.

The notation $\cdot$ and $:$ mean
\begin{align*}
p \cdot q = \sum_{i=1}^d p_i q_i \mbox{ and } \gamma : \phi = \mbox{trace}(\gamma\phi),
\end{align*}
where $p, q \in \bR^n$ and $\gamma, \phi \in \bS^n$.

As we will conclude, the HJB equation in this case is given by the following PPDE:
\begin{align}
\left\{\begin{array}{l}
\ds \widehat{H}(Y_t, Z_t,  \Delta_t V(Y_t, \cdot), \Delta_x V(Y_t, Z_t), \Delta_{xx} V(Y_t, Z_t))= 0,\\ \\
V(Y_T, Z_T) = g(Y_T),
\end{array}\right. \label{eq:path_dependent_hjb}
\end{align}
for any $(Y_T, Z_T) \in \Lambda_T^{n,k}$.\\

Here, the time derivative $\Delta_t$ is with respect to both variable $Y$ and $Z$:
\begin{align*}
\Delta_t V(Y_t, Z_t) = \lim_{\delta t \to 0^+} \frac{V(Y_{t, \delta t}, Z_{t, \delta t}) - V(Y_t, Z_t)}{\delta t},
\end{align*}
and the space derivative $\Delta_x$ is with respect to $Y$:
\begin{align*}
\Delta_x V(Y_t, Z_t) = \lim_{h \to 0} \frac{V(Y_t^h, Z_t) - V(Y_t, Z_t)}{h}.
\end{align*}

In a less compact notation, we could write the path\hyph dependent HJB equation (\ref{eq:path_dependent_hjb}) as
\begin{align*}
\left\{\begin{array}{l}
\ds \inf_{a \in \cA} \Big\{\Delta_t V(Y_t, Z_t^{a-z_t}) + H(Y_t, Z_t, \Delta_x V(Y_t, Z_t), \Delta_{xx} V(Y_t, Z_t), a) \Big\} = 0, \\ \\
V(Y_T, Z_T) = g(Y_T).
\end{array}\right. 
\end{align*}

\begin{remark}\label{rmk:predictable}
This remark will be the cornerstone of the proof of the Verification Theorem presented below. Notice that $V(Y_t, Z_t^h) = V(Y_t, Z_t)$, by the definition of the operator $\otimes$ given in (\ref{eq:otimes}). Denoting the functional derivatives with respect to the control $Z$ by $\Delta_\al$, we conclude $\Delta_\al V(Y_t, Z_t) = 0$, $\Delta_{\al \al} V(Y_t, Z_t) = 0$ and $\Delta_{x \al} V(Y_t, Z_t) = 0$. Hence, the dynamics of the control $\al$ will not impact the computations in the proof of the following theorem. This is similar to what \cite{rama_cont_fito_change_variable} assumed in order to consider functionals depending on the quadratic variation. These authors called such property \textit{predictability}. 

Moreover, if a smooth functional is predictable in a variable, then \textit{any space functional derivative} will be predictable in that variable. However, the \textit{time functional derivative} might not be predictable, in general. For example, the running integral functional $f(Y_t) = \int_0^t y_u du$ is predictable, but $\Delta_t f(Y_t) = y_t$ is not.
\end{remark}

\begin{theorem}[Verification Theorem]\label{thm:hjb_verification}

Suppose $V \in \bC^{1,2}$ solves the HJB equation (\ref{eq:path_dependent_hjb}). Under mild integrability conditions,
\begin{align*}
V(Y_t, Z_t) \leq J(Y_t, Z_t \otimes \al),
\end{align*}
for any $\al \in \bA$. Moreover, if there exists $\widehat{\al} \in \bA$ such that, for any $u \in [t,T]$,
\begin{align}\label{eq:optimal_control}
&\widehat{H}(X_u^{Y_t, Z_t \otimes \widehat{\al}}, (Z_t \otimes \widehat{\al})_u, \Delta_t V(X_u^{Y_t, Z_t \otimes \widehat{\al}}, \cdot), \Delta_x V, \Delta_{xx} V) \\
&= \Delta_t V(X_u^{Y_t, Z_t \otimes \widehat{\al}}, (Z_t \otimes \widehat{\al})_u) + H(X_u^{Y_t, Z_t \otimes \widehat{\al}}, (Z_t \otimes \widehat{\al})_u, \Delta_x V, \Delta_{xx} V, \widehat{\alpha}_u),\nonumber
\end{align}
then $V(Y_t, Z_t) = J(Y_t, Z_t \otimes \widehat{\al})$. All the functional derivatives in (\ref{eq:optimal_control}) are computed at $(X_u^{Y_t, Z_t \otimes \widehat{\al}}, (Z_t \otimes \widehat{\al})_u)$.

\end{theorem}

\begin{proof}
Let us apply the Functional It\^o Formula, Theorem \ref{thm:fif}, to $V(X_s^{Y_t, Z_t \otimes \al},$ $(Z_t \otimes \al)_s)$, for fixed $\al \in \bA$. Notice that the path $Z$ is frozen and that we are considering the control $Z_t \otimes \al$, which means we follow the path $Z_t$ as the control up to time $t$ (excluding it) and then $\al$ from $t$ to $T$. Moreover, since the functional derivatives of $V$ with respect to the control $\alpha$ are zero, it is not required to consider the dynamics of the control $\alpha$, see Remark \ref{rmk:predictable}. Furthermore, the time derivative is with respect to both variables. In the computation that follows we suppress the superscript of $X_s^{Y_t, Z_t \otimes \al}$ for a cleaner exposition:
\begin{align*}
&g(X_T) = V(X_T, Z_t \otimes \al) =  V(Y_t, Z_t) + \int_t^T \Delta_t V(X_u, (Z_t \otimes \al)_u) du \\
&+ \int_t^T \Delta_x V(X_u, (Z_t \otimes \al)_u) \cdot b(X_u, (Z_t \otimes \al)_u)du \\
&+ \int_t^T \Delta_x V(X_u, (Z_t \otimes \al)_u) \cdot \sigma(X_u, (Z_t \otimes \al)_u)dw_u\\
&+ \frac{1}{2} \int_t^T \Delta_{xx} V(X_u, (Z_t \otimes \al)_u) : \sigma\sigma^T(X_u, (Z_t \otimes \al)_u) du\\
&=V(Y_t, Z_t) + \int_t^T \left(\Delta_t V(X_u, (Z_t \otimes \al)_u) + H(X_u, (Z_t \otimes \al)_u, \Delta_x V, \Delta_{xx} V, \alpha_u) \right)du \\
&+ \int_t^T \Delta_x V(X_u, (Z_t \otimes \al)_u) \cdot \sigma(X_u, (Z_t \otimes \al)_u)dw_u - \int_t^T f(X_u, (Z_t \otimes \al)_u) du\\
&\geq V(Y_t, Z_t) + \int_t^T \widehat{H}(X_u, (Z_t \otimes \al)_u, \Delta_t V(X_u, \cdot), \Delta_x V, \Delta_{xx} V) du \\
&+ \int_t^T \Delta_x V(X_u, (Z_t \otimes \al)_u) \cdot \sigma(X_u, (Z_t \otimes \al)_u)dw_u- \int_t^T f(X_u, (Z_t \otimes \al)_u) du\\
&= V(Y_t, Z_t) + \int_t^T \Delta_x V(X_u, (Z_t \otimes \al)_u) \cdot \sigma(X_u, (Z_t \otimes \al)_u)dw_u \\
&- \int_t^T f(X_u, (Z_t \otimes \al)_u) du.
\end{align*}
Under integrability conditions and applying localization techniques, we might assume, without loss of generality, that the It\^o integral above is a martingale. Therefore, taking expectation on both sides, we conclude:
\begin{align*}
V(Y_t, Z_t) \leq \bE\left[g(X_T^{Y_t, Z_t \otimes \al}) + \int_t^T f(X_u^{Y_t, Z_t \otimes \al}, (Z_t \otimes \al)_u) du\right] = J(Y_t, Z_t \otimes \al).
\end{align*}

Taking the control $\widehat{\al}$ satisfying Equation (\ref{eq:optimal_control}), we find
\begin{align*}
V(Y_t, Z_t) = \bE&\left[g(X_T^{Y_t, Z_t \otimes \widehat{\al}}) + \int_t^T f(X_u^{Y_t, Z_t \otimes \widehat{\al}}, (Z_t \otimes \widehat{\al})_u) du\right] = J(Y_t, Z_t \otimes \widehat{\al}).
\end{align*}
as desired.
\end{proof}

\begin{remark}
We will see an interesting application of the Verification Theorem above in Section \ref{sec:example_control}, where we study the case of control with delay.
\end{remark}

\begin{remark}\label{rmk:x_control}

We would like to stress the difference between the case where the stochastic control problem exhibits a path\hyph dependent effect of control and state variables, which we are dealing with in this paper, and the case where there is only path\hyph dependent effect of state variables. In this case, it is not necessary to consider as variable of $V$ the path of the control, $Z_t$. It is enough to define
\begin{align*}
J(Y_t, \al) &= \bE\left[ g(X_T^{Y_t, \al}) + \int_t^T f(X_s^{Y_t, \al}, \alpha_s) ds\right],\\
V(Y_t) &= \inf_{\al \in \bA[t,T]} J(Y_t, \al),
\end{align*}
where $\bA[t,T]$ is the space of admissible controls on $[t,T]$. The HJB equation in this becomes
\begin{align*}
\left\{\begin{array}{l}
\ds \Delta_t V(Y_t) + \inf_{a \in \cA} \Big\{\frac{1}{2} \sigma \sigma^T(Y_t, a) : \Delta_{xx} V(Y_t) + b(Y_t, a) \cdot \Delta_x V(Y_t) + f(Y_t, a) \Big\} = 0, \\ \\
V(Y_T) = g(Y_T).
\end{array}\right. 
\end{align*}
See, for example, \cite{fournie_cont_thesis}, \cite{fito_hjb} or \cite{fito_hjb_ji}.

\end{remark}

\begin{remark}

It is obvious that if the dynamics of $x$ and the functionals $g$ and $f$ are path\hyph independent in the state variable and control, we find the classical HJB equation. Moreover, if the path\hyph dependence is only in the control, meaning that, for $h = b, \sigma, f$, 
\begin{align*}
h(Y_t, Z_t) = h(t, y_t, Z_t) \mbox{ and } g(Y_T) = g(y_T),
\end{align*}
the path\hyph dependent HJB Equation (\ref{eq:path_dependent_hjb}) becomes
\begin{align}
\left\{\begin{array}{l}
\widehat{H}(t, y, Z_t, \Delta_t V(t,y,\cdot), \partial_x V(t, y, Z_t), \partial_{xx} V(t, y, Z_t)) = 0,\\ \\
V(T, y, Z_T) = g(y),
\end{array}\right. \label{eq:only_control_path_dependent_hjb}
\end{align}
where $\partial_x$ is the usual derivative with respect to the state variable and
\begin{align*}
\widehat{H}(t, y, Z_t, q, p, \gamma) = \inf_{a \in \cA} \Big\{q(Z_t^{a - z_t}) &+ \frac{1}{2} \sigma \sigma^T(t, y, Z_t^{a-z_t}) : \gamma + b(t, y, Z_t^{a-z_t}) \cdot p \\
& + f(t, y, Z_t^{a-z_t}) \Big\}.
\end{align*}
It is worth noticing that $\Delta_t$ is still a functional derivative. More precisely,
\begin{align*}
\Delta_t V(t, y, Z_t) = \lim_{\delta t \to 0^+} \frac{V(t+\delta t, y, Z_{t, \delta t}) - V(t, y, Z_t)}{\delta t}.
\end{align*}

\end{remark}

\subsubsection{Delayed Control}\label{sec:example_control}

We will exemplify the results derived in the section above, mainly the path\hyph dependent HJB equation, by considering the delay type of path\hyph dependence in the control as in \cite{gozzi_delay_1}, see also \cite{gozzi_delay_2_I, gozzi_delay_2_II, delay_control_huang, delay_control_li, delay_control_alekal}. Namely, we will assume that the drift and the volatility are given by
\begin{align*}
b(t, y, Z_t) = c_0 y + b_0 z_t + \int_{-\tau}^0 b_1(u)z_{t + u}du, \quad \sigma(t,y, Z_t) = \sigma,
\end{align*}
where $b_1 \in L^2([-\tau,0];\bR)$ or, the more complicated case, dealt in \cite{gozzi_delay_2_I, gozzi_delay_2_II}, where $b_1$ is a measure. A very important example being the Dirac mass at $-\tau$, denoted by $\delta_{-\tau}$. In the measure case, we assume, without loss of generality, there is no Dirac mass at 0. As we will see below, differently than the aforesaid references, the framework proposed here can deal with both these situations without additional difficulty.

The Hamiltonian becomes
$$H(t, y, Z_t, p, \gamma, a) = \frac{\sigma^2}{2} \gamma + (c_0 y + b_0 a + \int_{-\tau}^0 b_1(u)z_{t + u}du) \ p + f(t, y, Z_t^{a - z_t}).$$

In order to get a complete characterization of the value functional (up to computing the solution of a system of PDEs), we consider the following linear\hyph quadratic example:
\begin{align*}
&c_0 = 0, \ b_0 = 1, \ b_1 = \delta_{-\tau}, \ f(t,y, Z_t) = \frac{z_t^2}{2} + \beta z_t y + \frac{\eps}{2} y^2 \mbox{ and } g(y) = c \frac{y^2}{2}.
\end{align*}
Hence
\begin{align*}
H(t,y,Z_t,p,\gamma, a) = \frac{\sigma^2}{2} \gamma + (a - z_{t-\tau})p + \frac{a^2}{2} + \beta a y + \frac{\eps}{2}y^2.
\end{align*}

\begin{remark} Notice that even though $b$ is not a smooth functional (since the delayed functional is not time functional differentiable), one would expect (as it is the case) that the value function is indeed smooth in the functional sense. The idea behind this fact is that the spot delayed dependence of $b$ generates a continuum delayed dependence on the value function, see Equation (\ref{eq:value_example}) below.

\end{remark}

We consider the following ansatz for the value functional, as it was examined, for instance, in \cite{delay_control_huang}:
\begin{align}\label{eq:value_example}
V(t, y, Z_t) &= \frac{1}{2} F_0(t) y^2 + y \int_{t-\tau}^t F_1(t,\theta-t)z_{\theta} d\theta \\ \nonumber
&+ \int_{t-\tau}^t \int_{t-\tau}^t F_2(t,\theta_1-t, \theta_2-t)z_{\theta_1} z_{\theta_2} d\theta_1 d\theta_2 + F_3(t),
\end{align}
where we assume that $F_2$ is symmetric in the last two variables as it is usually done in these problems: 
\begin{align*}
F_2(t,\theta_1, \theta_2) = F_2(t, \theta_2, \theta_1).
\end{align*}
We can compute the derivatives of $V$ explicitly and see that $V \in \bC^{1,2}$. The time functional derivative $\Delta_t$ would be more complicated, but for this ansatz, it may be verified that it is equivalent to taking derivative with respect to $t$:
\begin{align}
&\partial_x V = F_0(t) y + \int_{t-\tau}^t F_1(t, \theta-t) z_{\theta} d\theta,\\
&\partial_{xx} V = F_0(t),\\
&\Delta_t V = F_0'(t) \frac{y^2}{2} + y \Big( F_1(t,0)z_t - F_1(t,-\tau)z_{t-\tau}) \label{eq:time_derivative_example}\\
&+ \int_{t-\tau}^t \left(\frac{\partial F_1}{\partial t} - \frac{\partial F_1}{\partial \theta}\right)(t,\theta-t)z_{\theta} d\theta \Big)  \nonumber\\
&+ 2z_t\int_{t-\tau}^t F_2(t,\theta-t,0) z_{\theta} d\theta - 2z_{t-\tau}\int_{t-\tau}^t F_2(t,\theta-t,-\tau) z_{\theta} d\theta \nonumber \\
&+ \int_{t-\tau}^t \int_{t-\tau}^t \left(\frac{\partial F_2}{\partial t} - \frac{\partial F_2}{\partial \theta_1} - \frac{\partial F_2}{\partial \theta_2}\right)(t,\theta_1-t, \theta_2-t)z_{\theta_1} z_{\theta_2} d\theta_1 d\theta_2 + F_3'(t).\nonumber
\end{align}
Combining all derivatives into the modified Hamiltonian (\ref{eq:modified_hamiltonian}), we find that the terms that depend on the current control $a$ are:
\begin{align}
y F_1(t,0)a + 2a \int_{t-\tau}^t F_2(t,\theta-t,0) z_{\theta} d\theta + a p + \frac{a^2}{2} + \beta y a. \label{eq:alpha_example}
\end{align}
The infimum is then attained at
\begin{align*}
\widehat{a}(t,y, Z_t, p) = -\beta y - p - yF_1(t,0) - 2\int_{t-\tau}^t F_2(t,\theta-t,0) z_{\theta} d\theta,
\end{align*}
and the minimum value of the expression (\ref{eq:alpha_example}) is given by $-\widehat{a}(t,y, Z_t, p)^2/2$.
The HJB equation in this example becomes:
\begin{align}\label{eq:hjb_example_delay}
\left\{\begin{array}{l}
\ds\Delta_t V(t,y, Z_t^{-z_t}) + \frac{\sigma^2}{2} \partial_{xx} V(t,y, Z_t) - z_{t - \tau} \partial_x V(t, y, Z_t)\\ \\ 
\ds- \frac{1}{2} \widehat{a}^2(t,y, Z_t, \partial_x V(t, y, Z_t))  + \frac{\eps}{2}y^2 = 0,\\ \\
V(T, y, Z_T) = c \ds\frac{y^2}{2}.
\end{array}\right.
\end{align}
Notice that $\Delta_t V(t,y, Z_t^{-z_t})$ removes the terms that depend on $z_t$ in Equation (\ref{eq:time_derivative_example}). Additionally, the optimal control is given by
\begin{align*}
\begin{split}
\widehat{a}(t,y, Z_t, \partial_x V(t, y, Z_t)) &= -(F_0(t) + F_1(t,0) + \beta)y \\
&- \int_{t-\tau}^t \left( F_1(t, \theta-t) + 2F_2(t, \theta-t,0) \right) z_\theta d\theta.
\end{split}
\end{align*}

Combining all derivatives into HJB Equation (\ref{eq:hjb_example_delay}), we find
\begin{align*}
&\ds \frac{1}{2} y^2 \left(F_0'(t) - (F_0(t) + F_1(t,0) + \beta)^2 + \eps\right) \\
&\ds + y \left(- (F_1(t,-\tau) - F_0(t))z_{t-\tau}  \right.\\
&\ds -\frac{1}{2}(F_0(t) + F_1(t,0) + \beta)\int_{t-\tau}^t \left( F_1(t, \theta-t) + 2F_2(t, \theta-t,0) \right) z_\theta d\theta \\
&\left. \ds + \int_{t-\tau}^t \left(\frac{\partial F_1}{\partial t} - \frac{\partial F_1}{\partial \theta}\right)(t,\theta-t)z_\theta d\theta \right)  \\
&\ds + F_3'(t) + \frac{\sigma^2}{2} F_0(t) - z_{t-\tau}\int_{t-\tau}^t (2F_2(t, \theta-t, -\tau) - F_1(t,\theta - t))z_{\theta} d\theta\\
&\ds +\int_{t-\tau}^t \int_{t-\tau}^t \left(\frac{\partial F_2}{\partial t} - \frac{\partial F_2}{\partial \theta_1} - \frac{\partial F_2}{\partial \theta_2}\right)(t,\theta_1-t, \theta_2-t)z_{\theta_1} z_{\theta_2} d\theta_1 d\theta_2 \\
&\ds -\frac{1}{2} \int_{t-\tau}^t \int_{t-\tau}^t \left( F_1(t, \theta_1-t) + 2F_2(t, \theta_1-t,0) \right) \\ 
&\ds \left( F_1(t, \theta_2-t) + 2F_2(t, \theta_2-t,0) \right) z_{\theta_1} z_{\theta_2} d\theta_1 d\theta_2 = 0,
\end{align*}
with the following final conditions:
\begin{align*}
\left\{\begin{array}{l}
\ds F_0(T) = c,\\ \\
\ds F_1(T,\theta-T) = 0, \ \forall \ \theta \in (T-\tau, T),\\ \\
\ds F_2(T,\theta_1-T, \theta_2-T) = 0, \ \forall \ \theta_1, \theta_2 \in (T-\tau, T),\\ \\
\ds F_3(T) = 0.
\end{array}\right.
\end{align*}

Therefore, we find that, for any $t \in [0,T]$ and $\theta, \theta_1, \theta_2 \in (-\tau, 0)$,
\begin{align}\label{eq:system_PDE_control_begin}
&\left\{\begin{array}{l}
\ds F_0'(t) - (F_0(t) + F_1(t,0) + \beta)^2 + \eps = 0,\\ \\
\ds F_0(T) = c,
\end{array}\right. \\ \nonumber \\
&\left\{\begin{array}{l}
\ds \left(\frac{\partial F_1}{\partial t} - \frac{\partial F_1}{\partial \theta}\right)(t,\theta) \\ \\
\ds -\frac{1}{2}(F_0(t) + F_1(t,0) + \beta)\left( F_1(t, \theta) + 2F_2(t, \theta,0) \right) = 0, \\ \\
\ds F_1(t,-\tau) = -F_0(t),\\
\ds F_1(T,\theta) = 0,
\end{array}\right. \\ \nonumber \\
&\left\{\begin{array}{l}
\ds \left(\frac{\partial F_2}{\partial t} - \frac{\partial F_2}{\partial \theta_1} - \frac{\partial F_2}{\partial \theta_2}\right)(t,\theta_1, \theta_2)  \\ \\
\ds -\frac{1}{2}( F_1(t, \theta_1) + 2F_2(t, \theta_1,0) )( F_1(t, \theta_2) + 2F_2(t, \theta_2,0) ) = 0,\\ \\
\ds F_2(T,\theta_1, \theta_2) = 0,\\ \\
\ds F_2(t, \theta,-\tau) = F_2(t,-\tau, \theta) = -\frac{1}{2} F_1(t,\theta),
\end{array}\right. \\ \nonumber\\
\label{eq:system_PDE_control_end}
&\left\{\begin{array}{l}
\ds F_3'(t) + \frac{\sigma^2}{2} F_0(t) = 0,\\ \\
\ds F_3(T) = 0.
\end{array}\right.
\end{align}

In Figure \ref{fig:control_example}, we show the numerical solution of the PDE system above for the following parameters: $\beta = 1$, $\eps = 2$, $c=0$, $T=1$, $\tau = 0.05$ and $\sigma = 1$.

\begin{figure}[h!]
    \centering
    \includegraphics[width=0.7\linewidth]{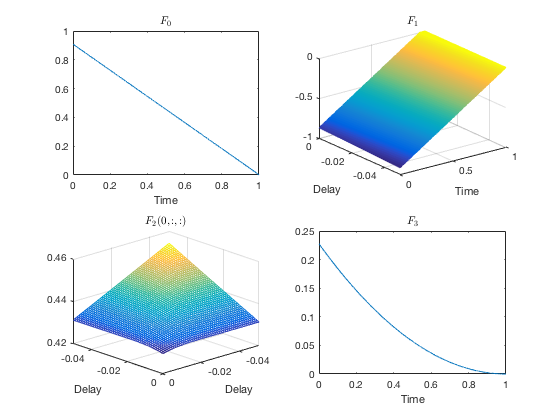}
    \caption{Numerical solution of the system PDEs (\ref{eq:system_PDE_control_begin})\hyph (\ref{eq:system_PDE_control_end})}
    \label{fig:control_example}
\end{figure}

\subsection{Stochastic Differential Games}\label{sec:games}

In this section, we will briefly analyze Stochastic Differential Games. Firstly, we present the general theory relating the game value function and a version of the HJB equation when there is path\hyph dependence in the control. Then, we exemplify the theory using the delayed stochastic differential game proposed in \cite{fouque_systemic_delay}.

Consider $N$ agents indexed by $i=1,\ldots,N$. These agents will act on a system whose state is described below:
\begin{align*}
\left\{\begin{array}{l}
dx_s^{Y_t, \al} = b(X_s^{Y_t, \al}, A_s) ds + \sigma(X_s^{Y_t,\al}, A_s)dw_s, \mbox{ if } s > t,\\ \\
X_t^{Y_t, \al} = Y_t,
\end{array}\right.
\end{align*}
where $Y_t \in \Lambda_t^m$, $m \in \bN$, $w$ is a $d$\hyph dimensional standard Brownian motion, $\al = (\al^1, \ldots, \al^N)$ with $\al^i$ being the $k_i$\hyph dimensional control chosen by agent $i$ taking values in $\cA^i$. Moreover,
\begin{align*}
(b, \sigma) : \Lambda^{m \times k} \longrightarrow \bR^m \times \bR^{m \times d},
\end{align*}
with $k = k_1 \times \cdots \times k_N$. The set of admissible controls of agent $i$ is denoted by $\bA^i$ and $\bA = \bA^1 \times \cdots \times \bA^N$. 
The agent $i$ chooses its own control $\alpha^i$ to minimize its own cost functional:
\begin{align*}
J^i(Y_t, \al) = \bE\left[g^i(X_T^{Y_t, \al}) + \int_t^T f^i(X_s^{Y_t, \al}, A_s)ds \right],
\end{align*}
where $g^i : \Lambda_T^m \longrightarrow \bR$ and $f^i : \Lambda^{m \times k} \longrightarrow \bR$ are his/hers terminal and running costs. We will define now the concept of equilibrium we will consider. The following notation will be used: for any $\al \in \bA$: $\al^{-i} = (\al^1, \ldots, \al^{i-1},$ $\al^{i+1}, \ldots, \al^N)$ and $(\al^{-i}, \widetilde{\al}) = (\al^1, \ldots, \al^{i-1}, \ \widetilde{\al}, \ \al^{i+1}, \ldots, \al^N)$.

\begin{definition}[Nash Equilibrium] An admissible strategy $\al^*$ is called a \textbf{Nash equilibrium} if, for any $i \in \{1, \ldots, N\}$ and $\widetilde{\al} \in \bA^i$, we have
$$J^i(Y_t, \al^*) \leq J^i(Y_t, (\al^{*^{-i}}, \widetilde{\al})).$$

Moreover, the Nash equilibrium can be further classified as
\begin{itemize}

\item \textbf{Open Loop}: $\alpha_t^{*^i} = \phi^i(W_t, x_0)$;

\item \textbf{Closed Loop}: $\alpha_t^{*^i} = \phi^i(X_t)$, 

\end{itemize}
for some functional $\phi^i$. See \cite{carmona_control}.

\end{definition}

In what follows, since we will use the HJB approach, we will be seeking a \textit{closed\hyph loop Nash equilibrium}.

Assuming that the other $N-1$ agents have already optimized their actions, denoted by
$\al^{*^{-i}} = (\al^{*^1}, \ldots, \al^{*^{i-1}},$ $\al^{*^{i+1}}, \ldots,\al^{*^N})$, the value functional for agent $i$ will be then given by
\begin{align*}
V^i(Y_t, Z_t) = \inf_{\al^i \in \bA^i} J^i(Y_t, Z_t \otimes (\al^{*^{-i}}, \al^i)).
\end{align*}

The DPP in this case (see Theorem \ref{thm:dpp}) becomes
\begin{align*}
V^i(Y_t, Z_t) = \inf_{\al^i \in \bA^i} &\bE\left[V\left(X_u^{Y_t, Z_t \otimes (\al^{*^{-i}}, \al^i )}, (Z_t \otimes (\al^{*^{-i}}, \al^i ))_u\right) + \right. \\
& \left. + \int_t^u f\left(X_s^{Y_t, Z_t \otimes (\al^{*^{-i}}, \al^i )}, (Z_t \otimes (\al^{*^{-i}}, \al^i ))_s\right) ds\right]
\end{align*}

Furthermore, under the assumptions of Theorem \ref{thm:hjb_verification}, we have a verification theorem for the following HJB Equation
\begin{align*}
\left\{\begin{array}{l}
\widehat{H^i}(Y_t, Z_t, \Delta_t V^i(Y_t, \cdot), \Delta_x V^i(Y_t, Z_t), \Delta_{xx} V^i(Y_t, Z_t)) = 0,\\ \\
V^i(Y_T, Z_T) = g^i(Y_T),
\end{array}\right.
\end{align*}
where
\begin{align*}
\widehat{H^i}(Y_t, Z_t, q, p, \gamma) = \inf_{a \in \cA^i} &\left\{q(Z_t^{(\alpha_t^{*^{-i}}, a) - z_t}) + \frac{1}{2} \sigma \sigma^T(Y_t, Z_t^{(\alpha_t^{*^{-i}}, a) - z_t}) : \gamma  \right. \\
&+ b(Y_t, Z_t^{(\alpha_t^{*^{-i}}, a) - z_t}) \cdot p + f^i(Y_t, Z_t^{(\alpha_t^{*^{-i}}, a) - z_t})\Big\}.
\end{align*}
Notice that $Z_t^{(\alpha_t^{*^{-i}}, a) - z_t}$ changes the control at time $t$ to $(\alpha_t^{*^{-i}}, a)$.

\subsubsection{Delayed Games}

We will now study the model introduced in \cite{fouque_systemic_delay}, where the authors proposes a stochastic differential game with delay in the control to analyze the systemic risk within a bank system.

Fix $m = d = N$, $k_i = 1$ and
\begin{align*}
&b(t, y, Z_t) = (z^1_t - z^1_{t -\tau}, \ldots, z^N_t - z^N_{t -\tau}) ,\\
&\sigma(t, y, Z_t) = \sigma I_N,\\
&f^i(t, y, Z_t) = \frac{(z^i_t)^2}{2} -  \beta z^i_t(\bar{y} - y^i) + \frac{\eps}{2} (\bar{y} - y^i)^2,\\
&g^i(y) = \frac{c}{2}(\bar{y} - y^i)^2,
\end{align*}
where $\bar{y} = \dfrac{1}{N}\sum_{i=1}^N y_i$ and $I_N$ is the identity matrix in $\bR^N$. Let us consider the same ansatz for the value functional as in \cite{fouque_systemic_delay},
\begin{align*}
V^i(t, y, Z_t) &= \frac{1}{2} E_0(t) (\bar{y} - y^i)^2 + (\bar{y} - y^i)\int_{t-\tau}^t E_1(t,\theta-t)(\bar{z}_{\theta} - z_{\theta}^i) d\theta \\
&+ \int_{t-\tau}^t \int_{t-\tau}^t E_2(t,\theta_1-t, \theta_2-t)(\bar{z}_{\theta_1} - z_{\theta_1}^i) (\bar{z}_{\theta_2} - z_{\theta_2}^i) d\theta_1 d\theta_2 + E_3(t).
\end{align*}
Assuming that $\alpha^j$ has been chosen, for $j \neq i$, by the reasoning outlined in Section \ref{sec:example_control}, the optimal control for the player $i$ is given by 
\begin{align*}
\widehat{a}^i(t,y, Z_t, p) = \beta(\bar{y} - y^i) - p_i - (\bar{y} - y^i)F_1(t,0) - 2\int_{t-\tau}^t F_2(t,\theta-t,0)(\bar{z}_{\theta} - z_{\theta}^i) d\theta.
\end{align*}
This is the same optimal control found in the aforesaid reference.

Assuming each player is following this strategy and noticing that $p_j$ in the formula for $\widehat{a}^j$ should be replaced by $\partial_{x_j} V^j$ (and not $\partial_{x_j} V^i$), we find that the HJB equation turns into
\begin{align}\label{eq:hjb_game}
\left\{\begin{array}{l}
\ds\Delta_t V^i(t, y, Z_t^{-z_t}) + \sum_{j=1}^N \left( \frac{\sigma^2}{2} \partial_{x_jx_j} V^i + (\widehat{a}_j(t, y, Z_t, \partial_{x_j} V^i) - z_{t-\tau}^j) \partial_{x_j} V^i\right)  \\ \\
\ds + \frac{1}{2}\widehat{a}^i(t,y, Z_t, \partial_{x_i} V^i)^2 - \beta \ \widehat{a}^i(t,y, Z_t, \partial_{x_i} V^i)(\bar{y} - y^i_t) + \frac{\eps}{2} (\bar{y} - y^i)^2 = 0,\\ \\
\ds V(T, y, Z_T) = \frac{c}{2}(\bar{y} - y^i)^2.
\end{array}\right.
\end{align}

Following the same arguments as in Section \ref{sec:example_control}, it is straightforward to find the same system of PDEs as in \cite{fouque_systemic_delay}. %Indeed, the system of PDEs in Section \ref{sec:example_control} is the limit as $N \to +\infty$ of the system in the aforesaid reference, as one would expect. 

\section{Conclusions}\label{sec:conclusions}

In this paper, we have studied stochastic control and differential games when there exists path\hyph dependence effect of the control of the agent in the dynamics of the state and in the running cost. We have analyzed the important example of delayed dependence. The framework used was the functional It\^o calculus, which has been proven to be an excellent tool to deal with complicated path\hyph dependence structures, see \cite{fito_greeks}. Although we have focused on delayed dependence, because of practical importance, there are no major impediments to examine more interesting structures. We hope this work will allow the consideration of different path\hyph dependent structures in other applications.

Compared to the theory of \cite{gozzi_delay_1} and \cite{gozzi_delay_2_I, gozzi_delay_2_II}, that deals with the delayed case, the method proposed here allows in principle very general path\hyph dependence in the controls. Moreover, it could be directly applied to (Dirac) measures, as it was done in Section \ref{sec:example_control}. 

Future research will be conducted to analyze viscosity solutions (existence and uniqueness) of the path\hyph dependent HJB derived here. Viscosity solution of similar PPDEs have been extensively studied in recent years, see for example \cite{fito_zhang_1, fito_touzi_ppde1, fito_touzi_ppde2}. Moreover, it would be interesting to apply the theory developed here to Stackelberg games, see for instance \cite{bensoussan_stackelberg}.

\section*{Acknowledgements}
I would like to thank J.P. Fouque for bringing such a interesting problem to my attention and for all the insightful conversations. I also thank J. Zhang and M. Mousavi for the innumerous helpful discussions and comments.

\bibliographystyle{plainnat}
%\bibliography{C:/Users/Yuri/Dropbox/BibTex/my_bib_tex}
%\bibliography{E:/Dropbox/BibTex/fito,E:/Dropbox/BibTex/stochastic,E:/Dropbox/BibTex/finance}
%\bibliography{/Users/yuri/Dropbox/BibTex/fito,/Users/yuri/Dropbox/BibTex/stochastic,/Users/yuri/Dropbox/BibTex/finance}

\end{document}